\documentclass[a4paper,12pt,reqno]{amsart}
\usepackage{amssymb}
\usepackage{enumerate}
\usepackage{ifthen}
\usepackage{graphicx}
\nonstopmode\numberwithin{equation}{section}
\newtheorem{definition}{Definition}[section]
\newtheorem{theorem}{Theorem}[section]
\newtheorem{corollary}{Corollary}[section]

\newtheorem{lemma}{Lemma}[section]

\newtheorem{remark}{Remark}[section]

\allowdisplaybreaks
    
\begin{document}
\title{ Continued $g$--fractions and geometry of bounded analytic maps}

\author{
Alexei Tsygvintsev
}
\address{
U.M.P.A, Ecole Normale Sup\'{e}rieure de Lyon\\
46, all\'{e}e d'Italie, F69364 Lyon Cedex 07
}
\email{atsygvin@umpa.ens-lyon.fr}

\bigskip
\begin{abstract}
In this work we study  qualitative properties of real analytic bounded maps. The main tool is approximation of real valued functions analytic in rectangular domains of the complex plane
 by   continued $g$--fractions of Wall \cite{W}.  As an application, the Sundman-Poincar\'e method in the  Newtonian three--body problem is revisited and applications to collision detection  problem are considered.
 \end{abstract}.  

\subjclass[2000]{ 37C30, 30E05, 11J70}

\keywords{Continued fractions,  real analytic functions, dynamical systems}
\maketitle
\pagestyle{myheadings}
\markboth{ 
Alexei Tsygvintsev
}{
$g$--fractions and geometry of bounded analytic maps
}

\section{The continued $g$--fraction representation for  real analytic  bounded functions }

By $R_{T,B}\subset \mathbb C $  we denote the open domain  which is the interior of the rectangle with vertices at the points $T +i\, B$, $T-i\, B$, $-T +i\, B$, $-T-i\, B$, $T$, $B>0$ (see Fig. \ref{fig3}).
  Let $\mathbb A_{M,T,B}$ be   the set of all functions  $f(z)$  analytic   in $R_{T,B}$, real valued  for $z\in I_T=(-T,T)$ and bounded in absolute value in  $R_{T,B}$ by $M>0$: 
\begin{equation}
| f(z)| <M,  \quad \forall \, z\in  R_{T,B}\,.   
\end{equation}

\begin{figure}[h]
\includegraphics[scale=0.400,angle=0]{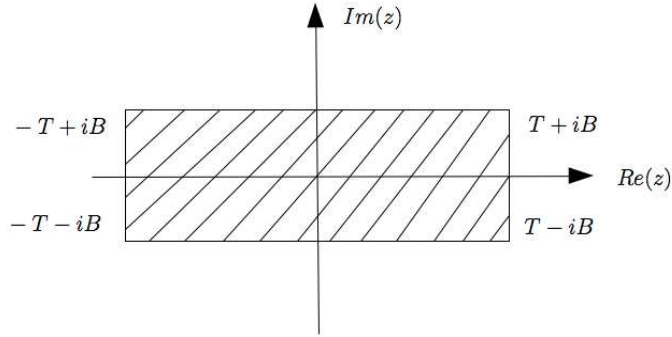} \caption{\label{fig3}  Domain  of analiticity  $R_{T,B}$   of $f(z)\in \mathbb A_{M,T,B}.$}
\end{figure}

Let  $\mathbb H=\mathbb C_-\cup \mathbb C_+ \cup (-1,+\infty)$ where $\mathbb C_+=\{ z\in \mathbb C\, : \, Im(z)>0\}$,  $\mathbb C_-=\{ z\in \mathbb C\, : \, Im(z)<0\}$.   We shall construct a conformal map between the open connected sets $R_{T,B}$ and $\mathbb H$ using the Jacobi elliptic function $\mathrm{sn}(z,k)$ and  the theta functions $\theta_2(z,q)$, $\theta_3(z,q)$ ( see  for  definitions   \cite{A1} ).

 Let 
 \begin{equation} \label{nome}
  q=e^{- \frac{\pi B}{T}}\,,
\end{equation}
be the  nome.

The corresponding real quarter-period   $K>0$ is defined as follows
\begin{equation} \label{f2}
K=\frac{\pi}{2}\theta_3 (0,q)^2\,.
\end{equation}
The elliptic modulus  $k\in (0,1)$  is given  by formula
\begin{equation} \label{kkk}
k=\frac{\theta _2(0,q)^2}{\theta_3 (0,q)^2}\,,
\end{equation}
and  defines the Jacobi elliptic function $\mathrm{sn}(z,k)$.

\begin{lemma}
Let 
\begin{equation}
\Phi(z)= \frac{2 \mathrm{sn}(  \frac{Kz}{T},k)}{1-\mathrm{sn}( \frac{Kz}{T},k)}\,.
\end{equation}
Then $\Phi \, :  R_{T,B} \to \mathbb H$ is conformal and  maps bijectively $(-T,T)$ to $(-1,+\infty)$, $\Phi(0)=0$.
\end{lemma}
The proof is straightforward  and easily follows from properties of $\mathrm{sn}(z,k)$ described in  \cite{Ach}, p. 119.
Let
\begin{equation} \label{domains}
\mathbb D_M=\{ z\in \mathbb C \,: | z|<M  \}, \quad \mathbb H_+=\{ z\in \mathbb C \,: Re(z)>0  \}\,.
\end{equation}
One verifies that   $m \,: \mathbb D_M \to \mathbb H_+$ defined by 
\begin{equation} \label{mapm}
m(z)=\frac{M+z}{M-z}\,,
\end{equation}
is conformal.

The composition 
\begin{equation} \label{f1}
F=m\circ f\circ \Phi^{-1}, \quad \mathrm{where} \quad f\in \mathbb A_{M,T,B}\,,
\end{equation}
is then a holomorphic function  in $\mathbb H$ such that  $F (\mathbb H)  \subset  \mathbb H_+$ and  $F \, : (-1,+\infty) \to  (0,+\infty)$.

According to Theorem of Wall \cite{W}, p. 279 there exist $\mu_0>0$ and the sequence of real numbers
\begin{equation}
 g_i\in[0,1],   \quad   i\geq 1\,,
 \end{equation}
 such that
\begin{equation}
F(z)=\mu_0\,  \sqrt{1+z} \,  \{g_1,g_2,...| z\}, \quad z\in \mathbb H\,,
\end{equation}
where
\begin{align} \label{f4}
 g(z)=\{ g_1,g_2,... | z  \}=\dfrac{1}{1}
\begin{array}{cc}\\+\end{array}
\dfrac{g_1z}{1}
\begin{array}{cc}\\+\end{array}
\dfrac{(1-g_1)g_2z}{1}
\begin{array}{cc}\\+\end{array}
\dfrac{(1-g_2)g_3z}{1}
\cdots\,,
\end{align}
is a continued $g$--fraction converging uniformly on compact sets  of $\mathbb H$ to an analytic function $g(z)$, $z\in \mathbb H$ (for applications of $g$--fractions see \cite{T2}--\cite{T4}).

\begin{remark}
$g(z)$ is a rational function of $z$  if and only if $g_k \in \{0,1\}$, for some $k\geq1$.  
\end{remark}

As follows from \eqref{f1}:  $f=m^{-1} \circ F \circ \Phi$  and hence the  following representation for $f(z)$  holds  
\begin{equation} \label{r}
f(z)=M\left(      1- \displaystyle \frac{2}{\mu_0  \sqrt{1+\Phi(z)}\,  \{g_1,g_2,... | \Phi(z)\} +1 }       \right), 
  \quad 
z\in R_{T,B}\,.
\end{equation}

To simplify  \eqref{r} we make  a rescaling  and obtain the new function $\phi$ given by 
\begin{equation}  \label{connection}
 \phi(z)=\frac{f({ z/ \alpha})}{M}, \quad \alpha=\frac{K}{T}\,,
\end{equation} 
 which is holomorphic in the  rectangle  $R_{K,\alpha B}$.

We note that  $ | \phi(z) |<1$, $\forall \, z\in R_{K,\alpha B}$.

Formula \eqref{r} then becomes 
\begin{equation} \label{main}
\phi(z)=\left(      1- \displaystyle \frac{2}{\mu_0  \sqrt{1+\eta(z)}\,  \{g_1,g_2,... | \eta(z)\} +1 }       \right), 
  \quad 
z\in R_{K,\alpha B}\,,
\end{equation}
where
\begin{equation} \label{eta}
\eta(z)= \frac{2 \mathrm{sn}(  z,k)}{1-\mathrm{sn}( z,k)}\,.
\end{equation}
 
The  map  $z \mapsto  \eta(z)$ is a bijection  between the real intervals $(-K,K)$ and  $(-1,+\infty)$, $\eta(0)=0$  what will be used later.

We define the truncated  continued $g$--fraction as  the $n$--order approximation of  \eqref{f4}:
 \begin{align} \label{truncated1}
\{g_1,g_2,...,g_n| z\}=\dfrac{1}{1}
\begin{array}{cc}\\+\end{array}
\dfrac{g_1z}{1}
\begin{array}{cc}\\+\end{array}
\dfrac{(1-g_1)g_2z}{1}
\begin{array}{cc}\\   \cdots \end{array}
\dfrac{(1-g_{n-1})g_nz}{1}\,,
\end{align}
which is a rational function of $z$ analytic in $\mathbb H$.

The next theorem  gives the {\it a priori} bounds for the $g$--fraction \eqref{f4}.
\begin{theorem} 
\emph{(\cite{T1})}
\label{T1}\\
\noindent a) Let $k=2n+1$, $n=0,1,...$,  then
\begin{equation}
A_k(z)\leq g(z)\leq B_k(z),\quad -1<z<+\infty \,,
\end{equation}
where
\begin{equation}
A_k(z)=\{g_1,g_2,...,g_k| z\}, \quad  B_k(z)=\{g_1,g_2,...,g_k,1 | z\}\,.
\end{equation}

\noindent b)  Let $k=2n$, $n=1,2,...$, then
\begin{equation}
A_k^+(z)\leq g(z)\leq B_k^+(z),\quad 0\leq z<+\infty,
\end{equation}
\begin{equation}
A_k^-(z)\leq g(z)\leq B_k^-(z),\quad -1<z<0,
\end{equation}
where 
\begin{equation}
A_k^+(z)=\{g_1,g_2,...,g_k,1 | z\}, \quad B_k^+=\{g_1,g_2,...,g_k| z\} \,,
\end{equation}
and $A_k^-=B_k^+$, $B_k^-=A_k ^+$.
\end{theorem}

Using the above formulas we write below  the rational {\it a priori} bounds for the $g$--fraction  \eqref{f4}  corresponding to  $k=1,2,3$:

\noindent Case $k=1$.

\begin{equation}  \label{rrr}
A_1(z)=\frac{1}{1+g_1z},\quad B_1(z)=\frac{1+(1-g_1)z}{1+z}\,.
\end{equation}

\noindent Case $k=2$.

\begin{equation} \label{bounds2}
A_2^+(z)=\frac{(1-g_1g_2)z+1}{(1+z)(g_1(1-g_2)z+1)}, \quad B_2^+(z)=\frac{g_2(1-g_1)z+1}{(g_1-g_1g_2+g_2)z+1}\,,
\end{equation}
\begin{equation}  \label{bounds22}
 A_2^-=B_2^+, B_2^-=A_2^+\,.
\end{equation}

\noindent Case $k=3$.

\begin{equation}
A_3(z)=\frac{(g_3+g_2-g_3g_2-g_2g_1)z+1}{g_1g_3(1-g_2)z^2+(g_3+g_2+g_1-g_3g_2-g_1g_2)z+1}\,.
\end{equation}

\begin{equation}
B_3(z)=\frac{g_2(1-g_3)(1-g_1)z ^2+(1+g_2-g_3g_2-g_1g_2)z+1}{(1+z)((g_1+g_2-g_3g_2-g_1g_2)z+1)}\,.
\end{equation}

The coefficients $g_p$  in formula \eqref{main}   are defined by 
\begin{equation} \label{54}
g_p=R_p(\phi(0),\phi'(0),\dots,\phi^{(p)}(0)),  \quad p\geq 1\,,
\end{equation}
with rational expressions  $R_p$ which can be  found  by calculation of derivatives of both sides of \eqref{main} and evaluating them at $z=0$. The  recurrent formulas for  $R_p$ can be derived from \cite{W}, p. 203.

In view of \eqref{connection},  $\phi_n=\phi^{(n)}(0)$ are  functions of  derivatives $f^{(n)}(0)$:
\begin{equation} \label{phin}
\phi_n=\frac{\alpha^{-n}}{M} f^{(n)}(0), \quad \alpha=K/T, \quad  n\geq 0\,.
\end{equation}

Below we give explicit expressions for $\mu_0$, $g_1$, $g_2$ in terms of $\phi_n$, $n=0,1,2$:

\begin{equation}
\mu_0=\frac{1+\phi_0}{1-\phi_0}\,,
\end{equation}

\begin{equation} \label{g1}
g_1=\frac{1}{2} \frac{1-\phi_0^2-2\phi_1}{1-\phi_0^2}\,,
\end{equation}

\begin{equation} \label{g2}
g_2=\frac{1}{2} \frac{(4\phi_1^2-2\phi_2-\phi_0+\phi_0^2+\phi_0^3-2\phi_2\phi_0-1)(1-\phi_0)   }{(2\phi_1-\phi_0^2+1)(2\phi_1+\phi_0^2-1)}\,.
\end{equation}

\begin{corollary} \label{rem}
As seen from \eqref{g1},  $f'(0)>0$  is equivalent  to   $g_1<1/2$; $f'(0)< 0$ is equivalent to   $g_1>1/2$ and  $f'(0)=0 \Leftrightarrow g_1=1/2$.
\end{corollary}

\section{ Bounds on the time of the first return }

Applying Theorem \ref{T1} to the $g$--fraction in  \eqref{r} one can derive  the {\it a priori} bounds for $f(z)$  holding inside the  interval $(-T,T)$.  Increasing the truncation order $n$ in \eqref{truncated1}  one obtains  more and 
 more precise information of qualitative character about $f(z)$  once the derivatives of $f(z)$ at $z=0$ are known.
In particular,  if $f(z)$ is   a solution of a system of  analytic  differential equations,  one can find often  recurrent formulas to calculate derivatives $f^{(n)}(0)$ of all orders $n\geq 0$ and write the $g$--fraction representation \eqref{r}.

\begin{figure}[h] 
\includegraphics[scale=0.500,angle=0]{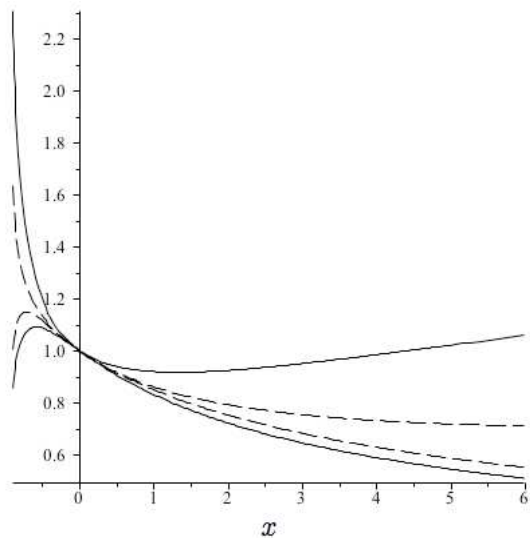} \caption{\label{fig2} Bounds  $ r(x) A_1(x) $,  $r(x)B_1(x)$ (bold line)   and  $r(x)A_2^{\pm}(x)$, $r(x)B_2^{\pm}(x)$ (dashed  line) for $x\in (-0.9,6)$,  $g_1=0.7$, $g_2=0.3$, $r(x)=\sqrt{1+x}$. }
\end{figure}

Our aim is  to estimate  the time of  return of $f(z)$   to the initial value $f(0)$ i.e to study the points $z_0\in (-T,T)$, $z_0\neq 0$ such that $f(z_0)=f(0)$.   To do this  we will  use the  {\it a priori} bounds  \eqref{rrr} applied to the $g$--fraction in formula  \eqref{main}. For $p=2k+1$ one obtains:

\begin{equation} \label{main1}
\left(      1- \displaystyle \frac{2}{\mu_0  \sqrt{1+\eta(z)}  A_p(\eta(z)) +1 }       \right)  \leq \phi(z) \leq \left(      1- \displaystyle \frac{2}{\mu_0  \sqrt{1+\eta(z)}  B_p(\eta(z)) +1 }       \right)\,,
\end{equation}
for $z\in (-K,K)$. 

If $p=2k$ then 
\begin{equation} \label{main2}
\left(      1- \displaystyle \frac{2}{\mu_0  \sqrt{1+\eta(z)}  A_p^+(\eta(z)) +1 }       \right)  \leq \phi(z) \leq \left(      1- \displaystyle \frac{2}{\mu_0  \sqrt{1+\eta(z)}  B_p^+(\eta(z)) +1 }       \right)\,,
\end{equation}
for $z\in (0,K)$, and

\begin{equation} \label{main21}
\left(      1- \displaystyle \frac{2}{\mu_0  \sqrt{1+\eta(z)}  A_p^-(\eta(z)) +1 }       \right)  \leq \phi(z) \leq \left(      1- \displaystyle \frac{2}{\mu_0  \sqrt{1+\eta(z)}  B_p^-(\eta(z)) +1 }       \right)\,,  
\end{equation}
for $z\in (-K,0]$.

\begin{definition}
We denote by $\mathbb A^{(k)}_{M,T,B}\subset \mathbb A_{M,T,B}$, $k=1,2,...$ the set of functions for which the $g$--fraction representation \eqref{r} satisfies the condition
\begin{equation}
g_i\not \in \{0,1 \}, \quad \forall \,  i=1,\dots,k\,.
\end{equation} 
\end{definition}

In the next theorem,  for a given $f\in \mathbb A^{(1)}_{M,T,B}$, we will  describe a neighborhood of origin in which $z=0$ is the only solution of $f(z)=f(0)$.

\begin{theorem} \label{theorem1}
Let $f(z)\in \mathbb A^{(1)}_{M,T,B}$, $f'(0)\neq 0$  where $g_1$ is defined by \eqref{nome}, \eqref{f2}, \eqref{phin} and  \eqref{g1} as a function of  $M,T,B$, $f(0)$, $f'(0)$.  
Let  $\tau \in (-T,T)$, $\tau\neq 0$  be the point such that $f(\tau)=f(0)$.  We define
\begin{equation} 
\tilde \tau=\frac{T}{K} \, |  \mathrm{sn}^{-1}  \left(g,k  \right)   | = \frac{T}{K} \, \left |  \int _0^g \, \frac{dt}{\sqrt{1-t^2} \, \sqrt{1-k^2t^2}  } \right |  \,,
\end{equation}
where
\begin{equation}
 g=\frac{1-2g_1}{g_1^2+(1-g_1)^2}\in(-1,1)\,,
\end{equation}
and $k$, $K$ are given by \eqref{kkk} and  \eqref{f2}.

Then   $ 0<\tilde \tau <T$ and 
\begin{equation} 
| \tau | \geq \tilde \tau \,.
\end{equation}
\end{theorem}  

\begin{proof}
One considers   \eqref{main1} with $p=1$.  We have  $A_1(0)=B_1(0)=1$  and define  $t_{1}$,  $t_2$  as non-zero  solutions  of the following algebraic equations
\begin{equation} \label{equations}
\sqrt{1+t_1}A_1(t_1)=1, \quad \sqrt{1+t_2}B_1(t_2)=1, \quad t_1,t_2\in (-1,+\infty)\,.
\end{equation}

Simple algebraic calculations show that the only  solutions satisfying \eqref{equations}  are  given by
\begin{equation} 
  t_1=\frac{1-2g_1}{g_1^2}, \quad t_2=\frac{2g_1-1}{(1-g_1)^2}\,,
\end{equation}
which are related by
\begin{equation}
\frac{1}{t_1}+\frac{1}{t_2}=-1\,.
\end{equation}
Since $\eta(z)$ is a bijection between the intervals $(-K,K)$ and $ (-1,+\infty)$ there exist unique real numbers  $T_1$, $T_2\in (-K,K)$ satisfying the following equations:
\begin{equation}
 \eta(T_{1})=t_{1}, \quad \eta(T_{2})=t_{2}\,.
\end{equation}
As easily seen from  \eqref{eta}: $T_1=-T_2$ and 
\begin{equation} \label{elliptic}
\eta^{-1}(y)= \int _0^{  \frac{y}{2+y}     } \, \frac{dt}{\sqrt{1-t^2} \, \sqrt{1-k^2t^2}  } , \quad y\in (-1,1)\,.
\end{equation}

  We define
\begin{equation}
\tilde \tau=\frac{T }{K}  | T_1 | =\frac{T }{K}  | \eta^{-1}(t_1) |     \in (0,T)\,.
\end{equation}
The proof  of Theorem \ref{theorem1}   follows  therefore  directly from \eqref{main1} and  \eqref{equations}.
\end{proof}

The next result shows that   $f(z)\in  \mathbb A^{(2)}_{M,T,B}$, under some conditions on derivatives $f^{(p)}(0)$, $p=0,1,2$,  always returns to the initial value $f(0)$  in the interval $(-T,T)$  i.e admits the oscillatory property.
\begin{theorem} \label{return}
Let $f(z)\in  \mathbb A^{(2)}_{M,T,B}$,  $f'(0)\neq 0$  where  $g_1,g_2$ are defined by formulas  \eqref{nome}, \eqref{f2}, \eqref{phin}  and   \eqref{g1}, \eqref{g2}.  We assume that one of the two following conditions (A) or (B)  holds
\begin{equation}
 g_1<1/2, \quad g_1^2-4(1-g_1)^2(1-g_2)g_2\geq 0 \quad (A)
 \end{equation}
 \begin{equation}
  g_1>1/2, \quad (1-g_1)^2-4g_1^2g_2(1-g_2)\geq 0  \quad (B)
\end{equation}
Then there exists $\tau\in (-T,T)$, $\tau \neq 0$  such that 
\begin{equation} \label{IN1}
 f(\tau)=f(0)\,.
\end{equation}
\end{theorem}

\begin{proof}
We  consider  \eqref{main2}  with   $p=2$  and define  the following real algebraic equations
\begin{align*}
 \sqrt{1+x}\,A_2^+(x)=1\,, \quad x \in (0,+\infty)\,, \quad \mathrm{(A_1)}\,,\\
 \sqrt{1+x}\,B_2^+(x)=1\,, \quad x \in (0,+\infty)\,,  \quad \mathrm{(B_1)}\,, \\
 \sqrt{1+x}\,B_2^-(x)=1\,, \quad x \in (-1,0)\,,  \quad \mathrm{(\tilde A_1)}\,, \\
 \sqrt{1+x}\,A_2^-(x)=1\,, \quad x \in (-1,0)\,.\quad \mathrm{(\tilde B_1)}\,,
\end{align*}
where $A_2^-=B_2^+$, $B_2^-=A_2^+$.

Making  the change of variables 
\begin{equation} \label{changef}
x=-1+t^2, \quad t\in \mathbb R\,,
\end{equation}
after some elementary  transformations, it is easy to show that  equations  ($A_1$), ($B_1$) are  equivalent respectively to quadratic equations ($A_2$) and    ($B_2$) given below
\begin{align*}
  P_1(t)=g_1(1-g_2)t^2-(1-g_1)t+g_1g_2=0,   \quad t\in \mathbb R,    \quad \mathrm{(A_2)}\\
P_2(t)=g_2(1-g_1)t^2-g_1t+(1-g_1)(1-g_2)=0,\quad t\in \mathbb R\,. \quad \mathrm{(B_2)}
\end{align*}
\begin{remark}
We notice that $P_2(t)$ is obtained  from $P_1(t)$  by transformation
\begin{equation}
g_i\mapsto 1-g_i, \quad  i=1,2\,.
\end{equation}
\end{remark}

The  polynomial $P_1(t)$ has  two real  roots   $t_1^{(1)},t_2^{(1)}\in \mathbb R$
\begin{equation} \label{t1}
t_1^{(1)}=\frac{1-g_1-\sqrt{D_1}}{2g_1(1-g_2)}, \quad t_2^{(1)}=\frac{1-g_1+\sqrt{D_1}}{2g_1(1-g_2)}, \quad t_1^{(1)}\leq t_2^{(1)}\,,
\end{equation}
if and only if the following condition holds 
\begin{equation} \label{D1}
D_1=(1-g_1)^2-4g_1^2g_2(1-g_2)\geq 0\,.
\end{equation}
$P_2(t)=0$ has  two real  solutions   $t_1^{(2)},t_2^{(2)}\in \mathbb R$
\begin{equation} \label{t2}
t_1^{(2)}=\frac{g_1-\sqrt{D_2}}{2(1-g_1)g_2}, \quad t_2^{(2)}=\frac{g_1+\sqrt{D_2}}{2(1-g_1)g_2}, \quad t_1^{(2)}\leq t_2^{(2)}\,,
\end{equation}
if and only if 
 \begin{equation} \label{D2}
D_2=g_1^2-4(1-g_1)^2(1-g_2)g_2\geq  0\,.
\end{equation}
Applying the  Vieta's formulas to polynomials $A_2$ and $B_2$, and taking into account that $g_i\in (0,1)$, $i=1,2$  one checks that:
\begin{equation} \label{VI}
t^{(i)}_j> 0, \quad  i,j=1,2\,.
\end{equation}

\noindent  {\it Case A.} Let $f'(0)>0 ( \Leftrightarrow g_1<1/2$). Then $f(z)$ is increasing function in the interval $(-\epsilon,\epsilon)$ for some small $\epsilon>0$. We assume that inequality  $D_2 \geq 0$ holds, so both roots $t^{(2)}_1$ and $t^{(2)}_2$ are real. One has $P_2(1)=1-2g_1>0$, so, in view of \eqref{VI},  either $0<t^{(2)}_1\leq t^{(2)}_2<1$ (a) or  $1<t^{(2)}_1\leq t^{(2)}_2$ (b).  One verifies with help of  \eqref{t2} that (a) is equivalent to  $L_2=g_1-2(1-g_1)g_2<0$ and (b) to  $L_2>0$.
Thus,  in view of \eqref{changef},  if (b) holds, the equation ($B_1$) will have  solution   $x=-1+{t^{(2)}_1}^2 \in (0, +\infty)$ and if (a) holds,  ($\tilde B_1$) will have solution  $x=-1+{t^{(2)}_2}^2 \in (-1, 0)$.\\

\noindent  {\it Case B.} Let $f'(0)<0 ( \Leftrightarrow g_1>1/2$). Then $f(z)$ is decreasing function in the interval $(-\epsilon,\epsilon)$ for some small $\epsilon>0$. We assume that inequality  $D_1 \geq 0$ holds, 
so both roots $t^{(1)}_1$ and $t^{(1)}_2$ are real. One has $P_1(1)=2g_1-1>0$, so, in view of \eqref{VI},  either $0<t^{(1)}_1\leq t^{(1)}_2<1$ (c) or  $1<t^{(1)}_1\leq t^{(1)}_2$ (d).  One verifies with help of \eqref{t1} that (c) is equivalent to  $L_1=1-g_1-2g_1(1-g_2)<0$ and (d) to  $L_2=1-g_1-2g_1(1-g_2)>0$.
Thus,  in view of \eqref{changef},  if (d) holds, the equation ($A_1$) will have  solution   $x=-1+{t^{(1)}_1}^2 \in (0, +\infty)$ and if (c) holds,  ($\tilde A_1$) will have  solution  $x=-1+{t^{(1)}_2}^2 \in (-1, 0)$ in view of \eqref{changef}.

As in proof of Theorem \ref{theorem1},  since $\eta(z)$ is a bijection of $(-K,K)$ to $(-1,+\infty)$, there exists unique real number  $\tilde T\in (-K,K)$
 satisfying equation $\eta(\tilde T)=x$ with $x\in(-1,+\infty)$ defined  above.

Let
\begin{equation}
\zeta=\frac{T \tilde T}{K}\in(-T,T)\,.
\end{equation}
Then, as follows from  \eqref{main2}, \eqref{main21},  there exists  $\tau$ satisfying  \eqref{IN1}.  One has   $\tau\in (0,\zeta)$ if $\zeta>0$ and  $\tau\in (\zeta,0)$ if  $\zeta<0$.   That finishes the proof.
\end{proof}

The next corollary contains explicit formulas and precises  the intervals containing the point of return $\tau$ defined in Theorem \ref{return}. 

\begin{corollary} We assume all conditions of Theorem \ref{return} being satisfied and  define four following subsets of $(0,1)^2$ (see Fig. \ref{fig1}):
\begin{align}
E=\{ (g_1,g_2)\in (0,1)^2 \,: \, D_2 \geq 0,  0< g_1<1/2,0<g_2<1/2 \}\,, \label{k1} \\
F=\{ (g_1,g_2)\in (0,1)^2 \,: \, D_2 \geq 0,  0< g_1<1/2,  1/2<g_2<1   \}\,, \label{k2} \\
G=\{ (g_1,g_2)\in (0,1)^2 \,: \, D_1 \geq 0,  1/2< g_1<1,0<g_2<1/2 \}\,, \label{k3} \\
H=\{ (g_1,g_2)\in (0,1)^2 \,: \, D_1 \geq 0,  1/2< g_1<1, 1/2<g_2<1 \} \,. \label{k4} 
\end{align}

Let $t^{(i)}_j$, $i,j=1,2$ be defined by \eqref{t1},  \eqref{t2} and $\eta^{-1}(z)$ by \eqref{elliptic}, then \\
\begin{align}
\zeta= \frac{T}{K} \eta^{-1}(-1+{t^{(2)}_1}^2)\in (0,T) \quad  \mathrm{if}  \quad  (g_1,g_2)\in E \\
\zeta= \frac{T}{K} \eta^{-1}(-1+{t^{(2)}_2}^2)\in   (-T,0)  \quad  \mathrm{if}  \quad  (g_1,g_2)\in F \\
\zeta= \frac{T}{K} \eta^{-1}(-1+{t^{(1)}_2}^2)\in (-T,0)  \quad  \mathrm{if}  \quad  (g_1,g_2)\in G \\
\zeta= \frac{T}{K} \eta^{-1}(-1+{t^{(1)}_1}^2)\in (0,T)  \quad  \mathrm{if}  \quad  (g_1,g_2)\in H 
\end{align}
For $\tau$ defined by \eqref{IN1} one has
\begin{equation}
\tau\in (0,\zeta) \quad  \mathrm{if}  \quad \zeta>0 \quad  \mathrm{and}   \quad \tau\in (\zeta,0) \quad  \mathrm{if}   \quad \zeta<0\,.   
\end{equation}
\end{corollary}

One easy verifies using the above formulas  that  $\zeta(g_1,g_2)$ $\to$ $0$ as  $g_1$ $\to$ $1/2$.

\begin{figure}[h]
\includegraphics[scale=0.500,angle=0]{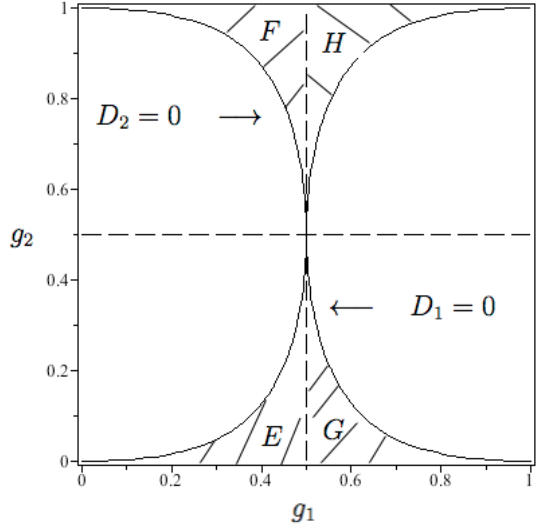} \caption{\label{fig1} Four domains  $E$, $F$, $G$, $H$ in the parameter  space $(g_1,g_2)\in (0,1)^2$ corresponding to the oscillatory behavior of $f(z)\in \mathbb A_{M,T,B }$.}
\end{figure}

\section{Functions bounded in the complex strip}

We denote   $\mathbb A_{M,\infty,B}$   the set  of functions  $f(z)$ satisfying the following conditions

\noindent 1. $f(z)$  is holomorphic in the infinite strip $S_B=\{ z\in \mathbb C\, : | Im(z) | <B \}$ , $B>0$.

\noindent 2.  $f(\mathbb R)\subset \mathbb R$.

\noindent 3.  $| f(z) | < M$, $M>0$,   $\forall $  $z\in S_B$.

\begin{figure}[h]
\includegraphics[scale=0.400,angle=0]{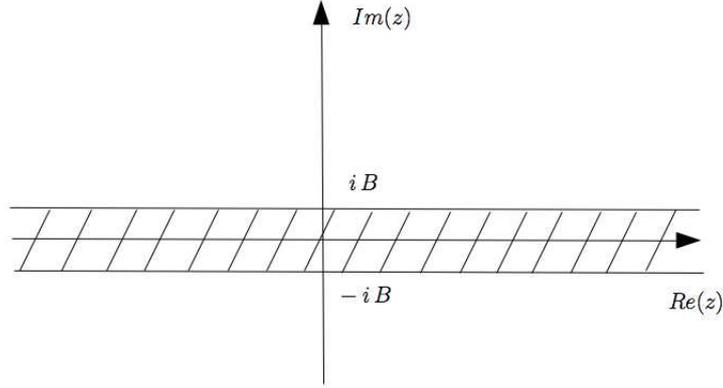} \caption{\label{fig4}  Domain  of analiticity  $S_{B}$   of $f(z)\in \mathbb A_{M,\infty,B}.$}
\end{figure}

The next result gives  a characterization of functions bounded in absolute value  in $S_B$
 with the help of $g$--fractions and can be considered as the limit case of  \eqref{r}  as  $T  \to +\infty$.  

\begin{theorem} \label{mainTh}
Let $f(z) \in \mathbb A_{M,\infty,B}$. Then   for  some $\mu_0>0$  and $g_k \in [0,1]$ , $k \geq 1$  one has
\begin{equation} \label{maingstrip}
f(z)=M \left(  1-\displaystyle \frac{2}{\mu_0 \exp\left(\frac{\pi z}{2B}\right)    \{ g_1,g_2,...| \exp \left( \frac{\pi z}{B} \right)-1  \} +1 }     \right).
\end{equation}
\end{theorem}

\begin{proof}
Let    $m \,: \mathbb D_M \to \mathbb H_+$ be defined by \eqref{mapm}.  We define the conformal map 
\begin{equation} \label{hhh}
 l(z)=\displaystyle \frac {B}{\pi}\log(1+z), \quad  l\,:   \mathbb H \to S_B\,.
\end{equation}

One verifies that the composition $F=m \circ f \circ l $ is holomorphic in $\mathbb H$ and $F(\mathbb H)\subset \mathbb H_+$
with $F(z)\in \mathbb  R$ for $z>-1$.
Thus, according to theorem of Wall \cite{W}, p. 279 $F$ can be written as follows
\begin{equation}
F(z)=\mu_0 \sqrt{1+z} \int_0^1 \displaystyle \frac{d \mu(u)}{1+zu}\,,
\end{equation}
for some nondecreasing real bounded function $\mu(u)$, $u\in (0,1)$ and $\mu_0>0$.

For $f=m^{-1}\circ F\circ l^{-1}$ one obtains the following formula
\begin{equation} \label{g12}
f(z)=M \left(  1-\displaystyle \frac{2}{\mu_0 \exp\left(\frac{\pi z}{2B}\right) \int_0^1\frac{d\mu(u) }{1+
(\exp \left( \frac{\pi z}{B} \right)-1)u}+1}     \right).
\end{equation}
The integral in \eqref{g12} can be transformed to the continued  $g$--fraction form  \cite{W}
\begin{equation}
   \int^1_0 \displaystyle \frac{d\mu (u)}{1+ (\exp \left( \frac{\pi z}{B} \right)-1)u}       
   =  \left  \{g_1,g_2,... | \exp \left( \frac{\pi z}{B} \right)-1 \right\}, \quad \mathrm{for \, \, some} \quad g_k\in[0,1]\,,
\end{equation}
that together with \eqref{g12} implies  \eqref{maingstrip}. The proof is finished.
\end{proof}

Let
\begin{equation}
\theta(z)=\frac{1}{M} f( Bz/ \pi )\,. 
\end{equation}

To calculate the coefficients $g_p$ in \eqref{maingstrip}  one has  formulas similar  to   \eqref{54}:
\begin{equation}
g_p=C_p(\theta(0),\theta'(0),..., \theta^{(p)}(0) ), \quad p\geq 1\,,
\end{equation}
with rational functions $C_p$  determined by calculation of derivatives  of both sides of \eqref{maingstrip} at $z=0$.   

Introducing
\begin{equation} \label{thetan}
 \theta_n=\theta^{(n)}(0)=\frac{1}{M}\, \frac{B^n}{\pi^n} f^{(n)}(0), \quad n \geq 0\,,
\end{equation}
 we provide below explicit formulas for $\mu_0$, $g_1$, $g_2$: 
\begin{equation}
\mu_0 =\frac{1+\theta_0} {1-\theta_0}\,,
\end{equation}

\begin{equation} \label{g1strip}
g_1= \frac{1}{2} \, \frac{ 1-4\theta_1-\theta_0^2} { 1-\theta_0^2  }   \,,
\end{equation}

\begin{equation} \label{g2strip}
g_2=  \frac{1}{2}\frac{(16\theta_1^2-8\theta_2-\theta_0+\theta_0^2+\theta_0^3-8\theta_2\theta_0-1)(1-\theta_0)} {(1-\theta_0^2+4\theta_1)(4\theta_1-1+\theta_0^2)}  \,.
\end{equation}

\begin{definition}
We denote by $\mathbb A^{(k)}_{M,\infty,B}\subset \mathbb A_{M,\infty,B}$, $k=1,2,...$ the set of functions for which the $g$--fraction representation \eqref{r} satisfies the condition
\begin{equation}
g_i\not \in \{0,1 \}, \quad \forall \,  i=1,\dots,k\,.
\end{equation} 
\end{definition}
The next result is analogous  to Theorem \ref{theorem1} and is proved in the similar way.

\begin{theorem} \label{theorem1strip}
Let $f(z)\in \mathbb A^{(1)}_{M,\infty,B}$, $f'(0)\neq 0$  where $g_1$ is defined by \eqref{thetan} and  \eqref{g1strip} as function of  $M,B$, $f(0)$, $f'(0)$.  
Let  $\tau \in  \mathbb R$, $\tau\neq 0$  be the point such that $f(\tau)=f(0)$.  We define
\begin{equation}
\tilde \tau=\frac{2B}{\pi} \left  |   \log \left( \frac{1-g_1}{g_1}\right) \right  |>0   \,.
\end{equation}
Then   
\begin{equation} 
| \tau | \geq \tilde \tau \,.
\end{equation}
 \end{theorem}

The following theorem is equivalent to Theorem \ref{return} for functions $f \in \mathbb A^{(2)}_{M,\infty,B}$ and has the  similar proof.

\begin{theorem} \label{2strip}
Let $f(z)\in \mathbb A^{(2)}_{M,\infty,B}$, $f'(0)\neq 0$ where   $g_1,g_2$ are defined by formulas  \eqref{thetan}  and   \eqref{g1strip}, \eqref{g2strip}.  We assume that  the point $(g_1,g_2)\in (0,1)^2$  belongs to one of the four regions $E$, $F$, $G$, $H$ defined by   \eqref{k1}-\eqref{k4}.

Let
\begin{align}
\zeta= \frac{2B}{\pi} \log(t^{(2)}_1) >0 \quad  \mathrm{if}  \quad  (g_1,g_2)\in E \\
\zeta= \frac{2B}{\pi} \log(t^{(2)}_2)<0  \quad  \mathrm{if}  \quad  (g_1,g_2)\in F \\
\zeta= \frac{2B}{\pi} \log(t^{(1)}_2)<0   \quad  \mathrm{if}  \quad  (g_1,g_2)\in G \\
\zeta= \frac{2B}{\pi} \log(t^{(1)}_1)>0  \quad  \mathrm{if}  \quad  (g_1,g_2)\in H 
\end{align}
where $t^{(i)}_j$, $i,j=1,2$ are defined as functions of $g_1,g_2$ by \eqref{t1} and  \eqref{t2}.

Then there exists $\tau\in \mathbb R^*$  such that 
\begin{equation} 
 f(\tau)=f(0)\,.
\end{equation}
and  $\tau\in (0,\zeta)$ if $\zeta>0$ and  $\tau\in (\zeta,0)$ if  $\zeta<0$.   

\end{theorem}

To conclude this section we will consider the case of functions   analytic and bounded  in a semi-infinite strip.
 Let $B>0$, $T>0$. We denote by $ S_{B,T}\subset \mathbb C$ the complex domain which is the interior of the  semi-infinite strip  formed by segments $t+i B$,$t-iB$,  $t\in [-T,+\infty)$ and $ -T+i t$, $t\in [-B,B]$ (see Fig. \ref{figfig}).

\begin{figure}[h]
\includegraphics[scale=0.400,angle=0]{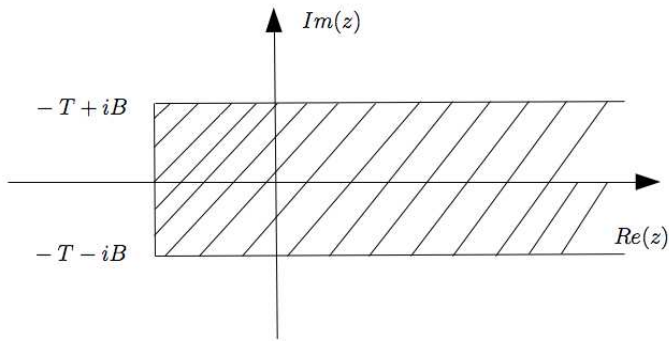} \caption{\label{figfig}  Domain  of analiticity  $S_{T,B}$   of $f(z)\in \mathbb D_{M,T,B}.$}
\end{figure}

Let  $\mathbb B_{M,T,B}$ be the set of functions $f(z)$ satisfying the following conditions:

\noindent 1. $f(z)$ is holomorphic in $ S_{B,T}$.

\noindent 2. $f(\mathbb R)\subset \mathbb R$.

 \noindent 3.  $| f(z) | < M$, $M>0$,  $\forall$  $z\in  S_{B,T}$.
 
It is straighforward to verify  that the following  function
\begin{equation}
L(z)=\frac{ \mathrm{cosh}\left(  \frac{\pi (z+T)}{B}   \right )-\mathrm{cosh}\left(\frac{\pi T}{B} \right) } { \mathrm{cosh}\left( \frac{\pi T}{B} \right )-1 }\,,
\end{equation} 
defines a conformal  map $L \, : S_{B,T} \to \mathbb H$ and maps bijectively $(-T,+\infty)$ to $(-1,+\infty)$, $L(0)=0$.  One can formulate now  the result analogous to  \eqref{r} and  \eqref{maingstrip}:

\begin{theorem}
Let $f(z) \in \mathbb B_{M,T,B}$. Then   for  some $\mu_0>0$  and $g_k \in [0,1]$ , $k \geq 1$  one has
\begin{equation} 
f(z)=M \left(  1-\displaystyle \frac{2}{\mu_0  \,   \sqrt{1+L(z)   }  \,    \{ g_1,g_2,...| L(z)  \} +1 }     \right)\,.
\end{equation}
\end{theorem}

\section{Applications}

We will apply the  results from previous sections  to the Newtonian  three--body problem, whose solutions in many situations are analytic functions in the strip along the real axis of the complex time plane.

We consider three mass points $P_1$, $P_2$, $P_3$ in $\mathbb R^3$ which attract each other according to the Newtonian law with finite positive masses $m_1$, $m_2$, $m_3$. Let $R_i = (x_i,y_i,z_i)$ be the position vector of $P_i$ and $r_{ij}$ the distance between it and mass $j$. One writes equations of motion as follows:
\begin{eqnarray}\label{3b}
m_i \frac{dR_i'}{dz}=-\sum_{j\neq i} m_i m_j\frac{R_i-R_j}{r_{ij}^3},\\
R_i'=\frac{dR_i}{dz}=(x_i',y_i',z_i'), \quad i=1,2,3\,,
\end{eqnarray}
which have the integral of energy:
\begin{eqnarray}
H=T+U=h=-\frac{m_1m_2m_3}{2\Gamma}K, \quad h,K=const\,, \\
T=\sum_{i=1}^3\frac{ m_i(x_i'^2+y_i'^2+z_i'^2)}{2},\\
U=-\frac{m_3m_2}{r_{32}}-\frac{m_1m_3}{r_{13}}-\frac{m_2m_1}{r_{21}},\\
\Gamma=m_1+m_2+m_3\,,
\end{eqnarray}
the first integrals of the impulse of the system:
\begin{eqnarray}
\sum_{i=1}^3m_i x_i=0, \quad   \sum_{i=1}^3m_i z_i=0, \quad  \sum_{i=1}^3m_i z_i=0\,,\\
\sum_{i=1}^3m_i x_i'=0, \quad  \sum_{i=1}^3m_i y_i'=0,  \quad \sum_{i=1}^3m_i z_i'=0\,,
\end{eqnarray}
and the first integrals of the angular momentum:
\begin{equation}
\sum_{i=1}^3m_i(x_iy_i'-x_i'y_i)=c_1,\quad \sum_{i=1}^3m_i(y_iz_i'-y_i'z_i)=c_2, \quad \sum_{i=1}^3m_i(z_ix_i'-z_i'x_i)=c_3\,,
\end{equation}
$c_1$, $c_2$, $c_3=const$.

\begin{figure}[h]
\includegraphics[scale=0.400,angle=0]{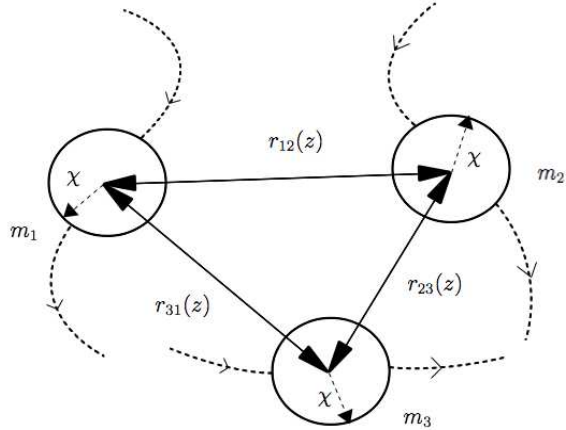} \caption{\label{fig6}    The planar three--planet problem.}
\end{figure}

We shall need the following result  due to Sundman \cite{Sun} which we write in a slightly  different form.
\begin{theorem} \label{Sudm} Let $J\subset \mathbb R$ be a connected open interval and $x_i(z)$, $y_i(z)$, $z_i(z)$, $i=1,2,3$  be a solution of the three-body problem \eqref{3b}  defined for  $z\in J$  and satisfying the following  inequalities
\begin{equation} \label{nc}
r_{32}(z)>2\chi, \quad r_{13}(z)>2\chi, \quad r_{21}(z)>2\chi, \quad \forall \, z\in J\,.
\end{equation}
Then  $\forall \, z_0 \in J$   the positions $R_i(z)$, $i=1,2,3$  are holomorphic  functions   in the complex disk 
\begin{equation}
 \Delta_{\chi}=\{ z\in \mathbb C\, :\,  |z-z_0|<B_{\chi}\}\,,
\end{equation}
where
\begin{equation} \label{BH}
B_{\chi}=\frac{ 1}{ 14  } \frac{2\chi}   {\sqrt{   \frac{28}{21}  \frac{\Gamma^2}{m\chi}  +\Gamma | K| }}, \quad m=\mathrm{min}\{m_1,m_2,m_3 \}\,,
\end{equation}
and satisfy  $\forall \, z\in  \Delta_{\chi} $ the inequalities
\begin{equation}
| x_i(z_0)-x_i(z)|< \chi/ 7, \quad | y_i(z_0)-y_i(z)|< \chi/ 7, \quad | z_i(z_0)-z_i(z)|< \chi/ 7, \quad i=1,2,3\,.
\end{equation}
\end{theorem}

Such a solution can be interpreted  as a collision free  motion for $z\in J$ of three planets each of radius $\chi>0$ ( see Fig. \ref{fig6}). The collision between bodies $P_i$ and $P_j$  at the moment of time $z\in \mathbb R$ happens if and only if  $r_{ij}(z)=2 \chi$.

\begin{lemma}\label{inv}
Let all conditions of Theorem \ref{Sudm} hold and $z_0 \in J$. Then the inverse mutual distances $r_{ij}^{-1}(z)$  are holomorphic functions  in $ \Delta_{\chi} $ and  bounded in the absolute value: 
\begin{equation} \label{principal}
| r_{ij}(z)^{-1}|<M_{\chi}=\frac{7}{4\chi},\quad  \forall\, z\in   \Delta_{\chi}\,.
\end{equation}
\end{lemma}

\begin{proof}
Let $<\,,\,>$ denotes the Euclidean scalar product in $\mathbb R^3$.  For some fixed $i<j$ we  introduce  $X(z)=(X_1(z),X_2(z),X_3(z))=r_i(z)$, $Y(z)=(Y_1(z),Y_2(z),Y_3(z))=r_j(z)$.
 Let $X(z)=X(z_0)+\tilde X(z)$, $Y(z)=Y(z_0)+ \tilde Y(z)$,  $z \in   \Delta_{\chi}$  where $\tilde X(z)=(\tilde X_1(z),\tilde X_2(z), \tilde X_3(z))$, $ \tilde Y(z)=( \tilde Y_1(z), \tilde Y_2(z), \tilde Y_3(z))$. Then
\begin{equation} \label{azza}
 | \tilde X_i(z) |    <\chi/7,\quad  |\tilde Y_i(z) |<\chi/7, \quad \forall \, z\in   \Delta_{\chi},  \quad i=1,2,3\,,
\end{equation}
as follows from Sundman's Theorem \ref{Sudm}.

Let $E=X(z_0)-Y(z_0)$, $R^2=<E,E>$, $K=\tilde X(z)- \tilde Y(z)$. Then $ |r_{ij}^2(z) |= |<E+K,E+K>|$.
Applying the triangular inequality one obtains for $z\in  \Delta_{\chi}$:
\begin{equation}
 |r_{ij}^2(z) | \geq R^2-2|<E,K>|-|<K,K>|> R^2-12R\frac{\chi}{7}-12\left (\frac{\chi}{7} \right)^2=\left( \frac{4\chi}{7} \right)^2\,,
\end{equation}
where we have used  the inequalities $R^2>4\chi ^2$ and \eqref{azza}.  That implies  \eqref{principal} and finishes the proof.
\end{proof}

We shall consider the case that $J=\mathbb R$ which corresponds   to the collision--free motion of three rigid spherical  bodies, each of radius $\chi>0$,   for $-\infty <z<+\infty$. In this case, according to Sundman's Theorem \ref{Sudm},  all inverse mutual distances
$r_{ij}^{-1}(z)$ are analytic functions in the complex infinite strip $S_{B_{\chi}}$.  As follows from Lemma \ref{inv},  $r_{ij}^{-1}(z)$ are bounded in $S_{B_{\chi}}$  in absolute value by $M_{\chi}$.  Moreover, $r_{ij}^{-1}(z)\in  \mathbb A^{(k)}_{M,\infty,B}$, $k=1,2$.  Indeed, in the opposite case the   $g$--fraction in \eqref{maingstrip} would  be rational one  of the form $\frac{a z+b}{cz +d}$, $a,b,c,d\in \mathbb R$. It is straightforward to verify that this case is eliminated with help of the equations of motion \eqref{3b}. 

 Thus,  the Theorems \ref{theorem1strip} and \ref{2strip} can be applied in this case. Below we  describe one possible  application to the collision problem: \\

{\it 
\noindent Let us consider the motion of  three spherical rigid bodies (planets) in $3$--dimensional Euclidean space  each of radius $\chi>0$ with  masses  $m_1$, $m_2$, $m_3$ interacting according to Newtonian low.  The motion is assumed to be collision free inside some small interval $z\in (-\epsilon,\epsilon)$, $\epsilon>0$. Let $r_{ij}(z)\in\{ r_{32}(z),r_{13}(z),r_{21}(z) \}$ be one of three mutual distances.   We put  $f(z)= r_{ij}^{-1}(z)$. One calculates $f(0), f'(0), f''(0)$ and finds  the upper and lower bounds on the time of the first return of $f(z)$ to its initial value $f(0)$ using Theorems  \ref{theorem1strip} and \ref{2strip}  with $B=B_{\chi}$, $M=M_{\chi}$ defined by \eqref{BH}, \eqref{principal}.  Then, if  these bounds are not fulfilled   for the  observed motion, then  there has to be a  collision between two  of bodies $P_1$, $P_2$, $P_3$  for some negative or positive value of time $z_0\in (-\infty,+\infty)$ i.e $r_{kl}(z_0)=2\chi$ for  some $(kl)\in \{(13),(21),(32) \} $. }

Indeed, let us assume that  the  bounds on the time of the first return given by Theorems  \ref{theorem1strip}, \ref{2strip}  and applied to $f(z)= r_{ij}^{-1}(z)$ with $B=B_{\chi}$, $M=M_{\chi}$ are not fulfilled.    Then  the condition  \eqref{nc} is not satisfied for some $z_0\in (-\infty,+\infty)$   and thus at  least   two  bodies collide  at the the moment of time  $z=z_0$.

\begin{center}
{\bf Acknowledgments}
\end{center}

The authors are grateful to the reviewer's valuable comments that improved the manuscript.

\end{document}